\documentclass[a4paper,reqno]{amsart}
\usepackage{amsmath,amsthm,amssymb,amsfonts,xspace}
\usepackage[T1]{fontenc}
\usepackage[utf8]{inputenc}
\usepackage[english]{babel}
\usepackage{amsxtra}
\usepackage{enumerate}
\usepackage{verbatim}
\usepackage{color}
\usepackage[mathscr]{eucal}

\usepackage[bookmarksnumbered,colorlinks]{hyperref}
\usepackage[colorinlistoftodos,italian]{todonotes}
\usepackage[normalem]{ulem}
\usepackage{tensor}

\newcommand{\N}{\mathbb N}

\newcommand{\R}{\mathbb R}

\newtheorem{Theorem}{Theorem}

\newtheorem{Proposition}{Proposition}
\theoremstyle{definition}
\newtheorem{Definition}{Definition}
\newtheorem{Remark}{Remark}

\def\bal#1\eal{\begin{align}#1\end{align}}              
\def\baln#1\ealn{\begin{align*}#1\end{align*}}          
\def\bml#1\eml{\begin{multline}#1\end{multline}}        
\def\bmln#1\emln{\begin{multline*}#1\end{multline*}}  
\def\bga#1\ega{\begin{gather}#1\end{gather}}
\def\bgan#1\egan{\begin{gather*}#1\end{gather*}}

\newcommand{\beq}{\begin{equation}}
\newcommand{\eeq}{\end{equation}}
\newcommand{\bere}{\begin{Remark}}
\newcommand{\ere}{\end{Remark}}
\newcommand{\bpr}{\begin{Proposition}}
\newcommand{\epr}{\end{Proposition}}

\def\br#1\er{\textcolor{black}{#1}} 

\begin{document}

\title[Analyticity of static solutions  of a field equation in Finsler gravity]{On the analyticity of static solutions  of a field equation  in Finsler gravity}
\author[E. Caponio]{Erasmo Caponio}
\address{Department of Mechanics, Mathematics and Management, \hfill\break\indent
	Politecnico di Bari, Via Orabona 4, 70125, Bari, Italy}
\email{caponio@poliba.it}

\author[A. Masiello]{Antonio Masiello}
\address{Department of Mechanics, Mathematics and Management, \hfill\break\indent
	Politecnico di Bari, Via Orabona 4, 70125, Bari, Italy}
\email{antonio.masiello@poliba.it}


\begin{abstract}
	It is well-known  that  static vacuum solutions of Einstein equations are  analytic in suitable coordinates. We ask here for an extension of this result in the context of Finsler gravity. We consider  Finsler spacetimes that retain several properties of static Lorentzian spacetimes, \br are  Berwald  \er and have vanishing Ricci scalar.
\end{abstract}
\maketitle

\section{Introduction}

In a couple of papers appeared in 1970~\cite{Muelle70,Mu70}, H. M\"uller zum Hagen  proved that  on any \br $C^3$  static or stationary spacetime which is a vacuum solution of the Einstein equations \er  there exists an appropriate analytic atlas such that \br the metric coefficients of the solution are \er also analytic. Our aim in this paper is to investigate if this result can be extended  to \br static \er Finsler spacetimes \br of Berwald type. \er 
This goal forces us to analyse  at least three~questions:

\begin{quote}

What is the convenient  definition of  a Finsler spacetime?
	
\noindent What the one of a static Finsler spacetime?
	
\noindent What the field equations extending  Einstein equations? 

\end{quote}

We will then consider each of the above  questions in the next three sections.  Section~\ref{sec5} will be devoted to the extension of  M\"uller zum Hagen's result to a static \br Berwald \er~spacetime.

\section{On the Definition of a Finsler~Spacetime}
The idea of replacing the Lorentzian norm  of a spacetime by a function  positively homogeneous on the velocities   goes back to G. Randers~\cite{Rander41}. He introduced a complex-valued norm 
$F(x,y)=\sqrt{h_x(y,y)} +\omega_x(y)$,
where $h$ is a Lorentzian metric and $\omega$ a one-form on a four-dimensional manifold $\tilde M$, that could  take more into account the asymmetries of the physical world, in~particular  the {\em ``uni-direction of timelike intervals''}. 
After decades,  Lorentz--Finsler  norms, eventually defined only on a  cone sub-bundle $ \mathcal A$ of $T\tilde M$ in order to avoid complex and/or negative values, appeared again in the work of G. S. Asanov (see~\cite{Asanov85} and the references therein) about general relativity and gauge field theory. Afterwards, they have been  considered in the study of multirefringence models  in optics~\cite{SkaVis11a},  in~the  classical limit of modified dispersion relations encompassing Lorentz violation in quantum gravity and in the Standard Model Extension (see, e.g.,~\cite{GiLiSi06, Vacaru11,  KoRuTs12, Russel15, Collad17}), in~ studies about  causality and superluminal signals (see, for~example,~\cite{ChLiWa12, PfeWoh12}). 
Actually,  Lorentz--Finsler norms  had already emerged   some years before  the work of Asanov  in  a paper by  H. Busemann~\cite{Busema67}, in~relation to the  local description  of  the spacetime according   to  an axiomatic definition called by the author  {\em timelike $G$-space}. Inspired by~\cite{Busema67}, J.K. Beem  in~\cite{Beem70} introduced the notion of an {\em indefinite  Finsler metric} as a non-reversible  fiberwise  positively homogeneous of degree 2 function $L=L(x,y)$, defined on the tangent bundle of $\tilde M$, whose fiberwise Hessian
\bal\label{fundtensor} 
\tilde g(x,y)[u,v]:=\frac 1 2 \frac{\partial^2}{\partial s\partial t}L(x,y+su+tv)|_{(s,t)=(0,0)}
\eal
$u,v\in T_x\tilde M$,  has index $1$ for all $x\in \tilde M$ and $y\in T_x\tilde M\setminus\{0\}$. This definition widely extends Lorentzian geometry, with~the fundamental  tensor $g$ replacing the  Lorentzian metric and the function \[F(x,\cdot):=\sqrt{|L(x,\cdot)|},\] 
giving a positively  homogeneous  Lorentz--Finsler norm at each tangent space $T_x\tilde M$. Notice that $F(x.\cdot)$ becomes  \br absolutely \er homogeneous if $L$ is reversible, i.e.,~$L(x,y)=L(x,-y)$.
Geodesics of $(\tilde M,L)$ connecting two points $p,q\in \tilde M$ can be defined as extremal curves  of the energy functional $\gamma\in \mathcal C_{pq}(\tilde M)\mapsto \int_0^1 L(\gamma,\dot\gamma)ds$, where $\mathcal C_{pq}(\tilde M)$ is the set of all the piecewise smooth curves $\gamma:[a,b]\to \tilde M$ such that $\gamma(a)=p$ and $\gamma(b)=q$. It is  soon realized that geodesics  must satisfy the conservation law $L(\gamma, \dot\gamma)=\mathrm{const.}$ and, as~a consequence,    world-lines of  freely falling particles are introduced \br kinematically \er as those timelike geodesics ($L(\gamma,\dot\gamma)<0$) parametrized with  $L(\gamma,\dot\gamma)=-1$.   
More recently, V. Perlick~\cite{Perlic06} used Beem's definition of the Finsler spacetime in order to extend Fermat's principle for light rays (i.e., geodesics satisfying $L(\gamma,\dot\gamma)=0$) between a point and a light source modelled  as a timelike curve. His Finslerian Fermat's principle recovers also some  results that had already appeared in the study of optics in an anisotropic medium and also  of  sound rays in an anisotropic elastic medium (see references in~\cite{Perlic06}). 

In~\cite{Beem70},  some $2$-dimensional examples of indefinite Finsler metrics $L$, reversible and not, are given where the set of lightlike vectors has more than two connected components.  E. Minguzzi~\cite{Minguz15} showed later that multiple light cones do not occur if $L$ is smooth on $T\tilde M\setminus 0$, reversible and  $\mathrm{dim}\,\tilde M\geq 3$. Anyway, for~non reversible and in particular for  functions  $L$ that are not smooth on the whole slit tangent bundle, multiple connected components are to be expected. This  fact   had led several authors to assume that only one of these connected components should be considered as a privileged one by the point of view of causality. The~choice can be done, e.g.,~by  prescribing a timelike, globally defined, vector field $Y$ and taking at each $x\in \tilde M$ the connected component which is the boundary of the set of timelike vectors containing $Y(x)$ (such as, for~example, in~\cite{GaPiVi12,Minguz15}) or by a priori restricting  $L$ to a cone sub-bundle  $\mathcal A$ of $T\tilde M$, like  in Asanov's   definition of a Finsler norm $F$ (such as, for~example, in~\cite{JavSan14a, AazJav16, Minguz17,Minguz17a}) or by looking only at the cone structure, without~considering as {\em fundamental} the  function $L$ (see~\cite{Fathi15, BerSuh18,  Minguz19, JavSan20}). In~some physical models, anyway,  indefinite Finsler metrics $L$  arise as the metrics invariant  under the action of the symmetry group considered and, in~general, they are  defined and smooth only on an open  cone sub-bundle of $T\tilde M$.  In~particular, this is the case of the Bogoslovsky metric (see, for~example,~\cite{BogGoe99, KoStSt09}). It was  observed that this  is  the metric that is preserved under the action of  the group of transformations of the so-called Very Special Relativity~\cite{GiGoPo07}.

Recently, a~definition of a Finsler spacetime  has been proposed~\cite{HoPfVo19} that encompasses definitions which  generalize Beem's one as those in~\cite{PfeWoh11, JavSan14a, AazJav16, Minguz17}.
The authors declare in~\cite{HoPfVo19} that their  definition does not include some classes of Finsler spacetimes studied   in~\cite{CapSta16, CapSta18} which can be seen as generalizations of standard static and stationary Lorentzian spacetimes  and that have already appeared in other papers~\mbox{\cite{Asanov98, LaPeHa12, SkaVis10, LiCha14}}. Thus,  it is worth  to relax slightly the definition in~\cite{HoPfVo19} in order to include~them.
\vspace{-3pt}
\begin{Definition}[Open  cone sub-bundle of $T\tilde M$]
	Let $\tilde M$ be a smooth connected manifold of dimension $n+1$ and $\tilde\pi: T\tilde M\to \tilde M$ its tangent bundle.  A~subset  $C\subset T\tilde M $ will be said an  {\em  open  cone sub-bundle} of $T\tilde M$, if~\vspace{-2pt}
	\begin{itemize}
		\item[(i)]$\tilde\pi(C)=\tilde M$;
		\item[(ii)] for all $x\in \tilde M$, $C_x:=T_x\tilde M\cap C$ is a pointed open  cone, i.e.,~$0\in C_x$, $C_x\setminus \{0\}$ is an open subset of $T_x\tilde M$ and  if $y\in C_x$ then $\lambda \br y\er\in C_x$ for each $\lambda> 0$; 
		\item[(iii)] $C_x$ varies smoothly with $x\in \tilde M$ meaning that $C_x\setminus \{0\}$ is defined by the union of the solutions of a finite number of  systems of inequalities  in the variable $y$
		\[
		\begin{cases}E_{1,k}(x,y)>0\\
		\ldots\\
		E_{m_k,k}(x,y)>0\end{cases}\]
		where, for~each $k\in\{1,\ldots,l\}$, $E_{1,k}, \ldots, E_{m_k,k}\colon T\tilde M\to \R$ are $m_k$ smooth functions on $T\tilde M\setminus 0$, positively homogeneous of degree $\alpha_{1,k}, \ldots, \alpha_{m_k,k}$ in $y$.
	\end{itemize}
	
	An open cone sub-bundle such that for all $x\in \tilde M$, $C_x$ is convex will be said a {\em  convex open cone sub-bundle}. Moreover, an~open cone sub-bundle without the zero section  will be called a {\em slit cone sub-bundle}. 
\end{Definition}
\begin{Remark}
	Notice  that we do not assume that $C_x$ is convex nor that it  is {\em salient}, i.e.,~that if $y\in C_x$ then $-y\not\in C_x$ (indeed  the open cone sub-bundle $C$ can be equal to $T\tilde M$). Being salient  is instead certainly true for a convex slit cone sub-bundle. Finally, notice that (iii) implies that the boundary of a fibre of an open cone sub-bundle is the union of a finite number of piecewise smooth hypersurfaces in $T_x\tilde M$. 
\end{Remark}

\begin{Definition}[Lorentz--Finsler metric and Finsler spacetimes]\label{LF}
	Let $C$ be an  open  cone sub-bundle of $T\tilde M$. A~{\em Lorentz--Finsler metric} on $\tilde M$ is a continuous function $L\colon C\to \R$ which~satisfies:\vspace{-3pt}
	\begin{itemize}
		\item[(i)] $L=L(x,y)$ is fiberwise  positively homogeneous of degree two, i.e. $L(x,\lambda y)$
		$= \lambda^2 L(x,y)$, for~all $x\in \tilde M$, $y\in C_x$ and all $\lambda\geq 0$;
		\item[(ii)] there exist a slit cone sub-bundle $A$ and an open cone sub-bundle $B$ such that $A \subset B\subset C$  and  $L$ is at least $C^1$ on $B$ and at least $C^4$ on $A$ with its {\em fundamental tensor} $\tilde g(x,y)$, defined as in \eqref{fundtensor}, being non-degenerate for all $(x,y)\in  A$;
		\item[(iii)] there exists a slit cone sub-bundle  $T\subset A\cap L^{-1}((-\infty,0)) $ such that its closure in  $L^{-1}((-\infty,0))$,  denoted by $T^A$, is a convex connected component of $L^{-1}((-\infty,0))$ contained in $B$ and, for~all $(x,y)\in T$, $\tilde g(x,y)$ has index $1$.  
	\end{itemize}
	
	A {\em Finsler spacetime} is a smooth finite dimensional manifold $\tilde M$ endowed with a Lorentz--Finsler metric. \vspace{-3pt}
\end{Definition}
This definition differs from the one in~\cite{HoPfVo19} essentially  because we  relax the condition that  there exists a connected component of $L^{-1}((-\infty,0))$, the~slit cone sub-bundle which represents all the  future-pointing timelike directions physically admissible, where $L$ is smooth (and $g$ has index $1$ on it). 
As done in~\cite{LaPeHa12}, a~quick and  elegant definition of a Lorentz--Finsler metric might consist in requiring that $L$ is $C^2$ a.e. on $T\tilde M$ with fundamental tensor having index $1$ a.e. on $T\tilde M$.  Anyway, it is preferable to control the lack of smoothness of $L$, hence  we allow the possibility that $\tilde g$ is not defined along some relevant future-pointing timelike direction  where $L$ remains differentiable at least once.   This requirement allows us to get  geodesics at least as weak extremal contained in $B$ of the energy functional. 
Let us recall that a piecewise $C^1$ curve $\gamma: [a,b]\to \tilde M$ is a continuous curve admitting a partition $\{t_i\}_{i\in\{0,\ldots, m\}}$, $m\in \N$, of~$[a,b]$ such that $\gamma|_{[t_{i-1},t_i]}$, for~all $i\in\{1,\ldots, m\}$, is $C^1$.  Let us denote by $\partial_xL:B\to T^*\tilde M$ and $\partial_yL:B\to T^*\tilde M$ the partial differentials of $L$ w.r.t. the first and the second  variable~respectively. \vspace{-3pt}
\begin{Definition}[Geodesics contained in $B$]
	Let $\gamma\colon [a,b]\to \tilde M$ be a piecewise $C^1$ curve such that $(\gamma,\dot \gamma)\subset  B$  then we say that $\gamma$ is a geodesic of $(\tilde M, L)$ if for any  
	\br piecewise $C^1$ vector field  $\zeta$ along $\gamma$ \er  with $\zeta(a)=\zeta(b)=0$ it holds
	\[\int_a^b(\partial_xL(\gamma,\dot\gamma)[\zeta]+\partial_yL(\gamma,\dot\gamma)[\dot \zeta])ds =0.\]  
\end{Definition}
Arguing as in \cite{CaJaSa14}, Prop. 2.51, it can be proved that the Legendre map of $L$ on $T^A$, i.e.,~$(x,y)\in T^A\mapsto  \partial_yL(x,y)[\cdot]\in T_x^*\tilde M$ is injective on $T^A$. Thus, if~$\gamma:[a,b]\to \tilde M$  is a geodesic such that $(\gamma, \dot\gamma)\subset T$  then, by~a standard argument about regularity of weak extremal and classical Finslerian computations (see, e.g.,~\cite{CaJaMa11}), we get that $\gamma$  must  be a $C^3$ curve satisfying the equation 
\beq\label{geoeq}
D^{\dot\gamma}_{\dot\gamma}\dot\gamma=0,
\eeq
where $D_{\dot\gamma}^{\dot\gamma}$ is the covariant derivative along $\gamma$ with reference vector $\dot\gamma$ defined by the Chern connection of $L$  which is well-defined on the open subset $A$ of $T\tilde M\setminus 0$ by (ii) of Definition~\ref{LF}. 
In local natural coordinates on $T\tilde M$, Equation  \eqref{geoeq} corresponds to $\ddot \gamma^i+\Gamma^i_{jl}(\gamma,\dot\gamma)\dot\gamma^j\dot\gamma^l=0$, where the components of the Chern connection are defined in Equation \eqref{Gamma} below. If~we know that  the Legendre map is injective on  $A$, the~same result holds for all weak extremals $\gamma$ such that  $(\gamma,\dot\gamma)\subset  A$. In~any case, if~we know that a  weak extremal  $\gamma$ is $C^1$  and $(\gamma,\dot\gamma)\subset A$ then it satisfies \eqref{geoeq} (and therefore it is actually $C^3$).\footnote{In some cases, smoothness or at least $C^1$-regularity of weak extremals hold; for example, this is true for some stationary splitting  Finsler spacetimes and for standard static Finsler spacetimes in next section, see, respectively, \cite{CapSta18}, Prop. A2 and \cite{CapSta16}, Th. 2.13.}
In particular, from~\eqref{geoeq} it follows that there exists a unique geodesic for each initial condition in $T$.  

As  Randers spacetime metrics show,  in~general $L$ will be not differentiable along null directions, i.e.,~along \br non-zero \er tangent vectors $(x,y)$ such that $L(x,y)=0$.   \br In order to have a definition for lightlike geodesics of a  non smooth  $L$, a~possible way is to require that, for~every initial null conditions $(x,y)$,  there~exists an open maximal interval $(-\epsilon, \epsilon)$ and a  $C^1$ curve $\gamma:(-\epsilon, \epsilon)\to \tilde M$, with~ $\gamma(0)=x$ and $\dot\gamma(0)=y$,   such that for every sequence $(x_k,y_k)\subset T\tilde M$ of initial conditions of  solutions $\gamma_k$  of \eqref{geoeq}, $\gamma$ is the limit in the $C^1$ topology of $\gamma_k$ (see \cite{HoPfVo19}, Def. 1-(iv), \cite{LaPeHa12}, Def. 1-(d)). A~more general way of defining them  (see~\cite{Minguz19, JavSan20}) is inspired by a well-known local property of lightlike geodesics in a spacetime: \er  \vspace{-3pt}
\begin{Definition}[Lightlike pregeodesics]\label{light}
	Let  $N\subset TM\setminus 0 $ be the set  of null directions, i.e.,~$N:=\{(x,y)\in T\tilde M\br \setminus 0\er: L(x,y)=0\}$. Let also $T^N$ be the closure of $T$ in $L^{-1}((-\infty,0])$. A~Lipschitz  curve $\gamma\colon[a,b]\to \tilde M$, such~that $(\gamma,\dot\gamma)\subset T^N\cap N$ a.e., is a {\em lightlike pregeodesic} if for any $s_0\in [a,b]$ there exists a neighbourhood $U$ of $\gamma(s_0)$ such that any two points in $\gamma([a,b])\cap U$ are not connected by any Lipschitz curve  $\alpha$ such that $(\alpha, \dot\alpha)\subset T^A$.  
\end{Definition}	\vspace{-3pt}
\br As a consequence of  \cite{JavSan20}, Theorem 6.6, we have that  \er if $T^N\cap N\subset A$ (i.e., $L$ is smooth on a neighbourhood of the   null directions in the boundary of $T^A$) then any lightlike pregeodesics \br  in the sense of Definition~\ref{light}, \er  is actually a geodesic, up~to reparametrization, i.e.,~it satisfies Equation~(\ref{geoeq}).	

\section{About the Notion of Stationary and Static Finsler~Spacetimes}
Let us recall the notion of a Killing vector field for a Finsler metric. We refer to~\cite{CapSta18} for details. Since ~$L$ is $C^1$ on the open cone sub-bundle $B$, we take  $B\subset T\tilde M\setminus 0$ as the base space instead of the slit tangent bundle which is usual in Finsler geometry (compare with~\cite{Lovas04}).  A~vector field $K$ on $\tilde M$ is a {Killing vector field} for $(\tilde M, L, B)$ if $K^c|_B(L)=0$, where  $K^c$ denotes the complete lift of $K$ to $T\tilde M$ (restricted to the open subset $B$). This is the vector on $T\tilde M$ whose local flow $\tilde \psi$ is given by $\tilde \psi_t(v)=(\psi_t(p), d\psi_t(p)[v])$, where $\psi$ is the flow of $K$, $p=\tilde \pi(v)$, $v\in TM$. Thus, if~$K$ is a Killing vector field then $L$ is invariant under the flow of $K^c$.
In natural local coordinates of $T\tilde M$, $K^c(L)$ is given by:\vspace{-3pt}
\[K^{c}(L)(x,y)=K^{h}(x)\frac{\partial L}{\partial x^{h}}(x,y)+\frac{\partial K^{h}}{\partial x^{i}}(x)y^{i}\frac{\partial L}{\partial y^{h}}(x,y),\]
for all $(x,y)\in B$ (the Einstein's sum convention is used in the above and in the following equations).
It is  not difficult to prove also that $K$ is a Killing field iff $K^c|_A$ is an infinitesimal generator of local $\tilde g$-isometries, i.e.,~for each $v\in A$ and for all $v_1,v_2\in  T_{\pi(v)}\tilde M$, we have  
\[\tilde g(\tilde \psi_t(v))\big[d \psi_{ t}(p)[v_1],d \psi_{t}(p)[v_2]\big]= \tilde g(v)[v_1,v_2].\]
for all $t\in I_p$, where $I_p\subset\R$ is an interval containing $0$  such that   the stages  $\psi_{t}$  are well-defined in a neighbourhood $U\subset \tilde M$ of $p=\pi(v)$ and $d\psi_t(p)[v]\in A$, for~each $t\in I_p$. 
Thus, the~Lie derivative of $\mathcal L_K\tilde g$ in $A$ vanishes. In~local natural coordinates on $T\tilde M$, this amount to say that 
\[K^{c}(\tilde g_{lj})+\frac{\partial K^h}{\partial x^l}\tilde g_{hj}+\frac{\partial K^h}{\partial x^j}\tilde g_{lh}=0.\]

\begin{Definition}[Stationary Finsler spacetime]
	A Finsler spacetime $(\tilde M, L)$ is said {\em stationary} if it is endowed with  a Killing vector field $K$ which is timelike, i.e., $L(x,K(x))<0$ for all $x\in \tilde M$.
\end{Definition}\vspace{-2pt}
In a Lorentzian manifold $(M,h)$, a~timelike Killing vector field $K$ is said {\em static} if $\mathrm{curl}K|_{\mathcal D}=0$, where $\mathcal D$ is the orthogonal distribution to $K$. Equivalently,  $K$ is static iff $\mathcal D$  is locally integrable; thus, for~each $p\in M$ there exist a spacelike hypersurface $S$, through $p$,  orthogonal to $K$, and~an open interval $I$ such that  the pullback of the metric $h$ by the flow of $K$, defined in $I\times S$, is given by  $-\Lambda dt^2+h_0$, where $t\in I$, $\partial_t$ is the pullback of $K$, $\Lambda=-h(K,K)$ and $h_0$ is the Riemannian metric induced on $S$ by $h$ (see \cite{One83}, Proposition~ 12.38). In~order to generalize this notion to Finsler spacetimes, requiring minimal  regularity assumptions on  the Lorentz--Finsler metric $L$, we consider  $\mathcal D$ as the distribution of dimension $n$ on $\tilde M$ defined pointwise by the kernel of the one-form $\partial_yL(x,K(x))[\cdot]$.
\begin{Definition}[Static Finsler spacetime]
	Let $(\tilde M, L)$ be a stationary Finsler spacetime endowed with a timelike Killing vector field $K$, such that $(x, K(x))\subset B$. We say that $K$ is \br {\em static} \er  if $\mathcal D:=\mathrm{ker}\big(\partial_yL(x,K(x))[\cdot]\big)$ is locally~integrable.
\end{Definition}
\begin{Remark}
	Let $U$ be a  vector field in $\tilde M$ such that $\big(x,U(x)\big )\in T^A$, for~all $x\in \tilde M$ and  $L\big(x,U(x)\big)=-1$. If~the distribution $\mathcal D=\mathrm{ker}\big(\partial_yL(x,U(x))[\cdot]\big)$ is integrable then its integral manifold can be used to define the {\em rest spaces} of the observer field $U$ (see the question posed in the final paragraph of~\cite{Pfeife19}). From~\cite{CapSta18}, Theorem 4.8,  if~$K:=\sigma U$ is also a Killing vector field, for~some positive function $\sigma$ on $\tilde M$, $B=T\tilde M $ and $L$ satisfies
	$L\big(x,K(x)\big)=L\big(x,-K(x)\big)$ and  $L\big(x, w\pm K(x)\big)=L(x,w)+L\big(x,K(x)\big)$ 
	for all $(x,w)\in \mathcal{D}$ then  $(\tilde M, L)$ is locally is isometric to a {\em standard static Finsler spacetime} (see the definition below).   
\end{Remark}	\vspace{-3pt}

Recall that we have assumed that in the open cone sub-bundle $B$,  $L$ is at least $C^1$, thus $\mathcal D$ above is well-defined. In~some Finsler spacetimes, this is the best possible  regularity level of $L$.  Consider, for~example, a~type of  {\em stationary splitting Finsler spacetime} introduced in~\cite{CapSta18}: 
assume that $\tilde M=\R\times M$ and  denote with $(t,x)$ points  in $\tilde M$ and by $(\tau,y)$ tangent vectors  of $T\tilde M$. 
Let \vspace{-3pt}
\beq\label{stationary}
L\big((t,x), (\tau,y)\big):=-\Lambda(x)\tau ^2+2\br b\er (x,y)\tau+F^2(x,y), \eeq  where $\Lambda$ is a smooth positive function on $M$, $\br b\er:TM\to \R$ a fiberwise positively homogeneous function and $F$ a Finsler metric on $M$.  Both $\br b\er $ and $F$ are assumed to be  smooth on $TM\setminus 0$; moreover,  the~fundamental tensor $g$ of $F$ (defined as in \eqref{fundtensor} with $F^2$ replacing $L$) is positive definite for any $(x,y)\in TM\setminus 0$ while the fiberwise Hessian of $\br b\er$ (defined analogously with $\br b\er$ in place of $L$) is positive semi-definite for all $(x,y)\in TM\setminus 0$. 
Let us denote by $\mathcal T$ the trivial line sub-bundle of $T\tilde M$ defined by the vector field $\partial_t$. In~this~case $C=T\tilde M$,  $B=T\tilde M\setminus \mathcal T$, $A=\big(T\tilde M\setminus \mathcal  T\big)\cap \big\{\big((t,x),(\tau,y)\big)\in T\tilde M: \tau>0\big\}$ (see \cite{CapSta18}, Prop. 3.3),
\baln
\lefteqn{T=}&\\
&\left\{\!\!\Big((t,x), (\tau,y)\Big)\in T\tilde M:y\in T_{x}M\setminus \{0\},\ \tau>\frac{\br b\er(x,y)}{\Lambda(x)}+\sqrt{\frac{\br b\er^2(x,y)}{\Lambda^2(x)}+\frac{F^2(x, y)}{\Lambda(x)}}\right\}\!\subset\!A\ealn
and $\partial_t$ is timelike and Killing.                                                                                                                                                                                                                                                                                                        
A Finsler spacetime of the type \eqref{stationary} has been considered  in~\cite{Rutz93}, where it has been shown to be a solution of the field equation $R=0$  (see next section). In~that paper, $L$ is a Finsler perturbation of the Schwarzschild metric, indeed $F$ is the norm of the Riemannian metric in the spacelike base and $\br b\er$ is a function conformal to the norm of  the standard Riemannian metric  on $S^2$,
see Equation~(40) in~\cite{Rutz93}:
\baln
\lefteqn{L\big((t,r, \theta,\varphi),(\tau, y_r,y_\theta,y_\varphi)\big):=}&\\ &\quad- \left(1-\frac{ 2m}{r}\right)\tau^2+\epsilon \left(1-\frac{ 2m}{r}\right)\tau\sqrt{y_\theta^2+\sin^2\theta y_\varphi^2}+\frac{y_r^2}{1-\frac{2m}{r}}+r^2(y_\theta^2+\sin^2\theta y_\varphi^2),
\ealn
where $\epsilon$ is a perturbation parameter.                                                                                                                                                                                                                                                                                                       
A particular case in type  \eqref{stationary}, is when $\br b\er$ is equal to a one-form $\omega$ on $M$. In~such a case,  $C=B=T\tilde M$ (i.e., $L$ is of class $C^1$ on $T\tilde M$) and $A=T\tilde M\setminus \mathcal T$. The~slit cone sub-bundle  $T$ of timelike future-pointing vector is defined as above with $\omega$ replacing $\br b\er$; now  there is also another  slit cone sub-bundle associated to $L$ which is 
\baln
\lefteqn{T^-=}&\\
&\left\{\Big((t,x), (\tau,y)\Big)\in T\tilde M:y\in T_{x}M\setminus \{0\},\ \tau<\frac{\omega_{x}(y)}{\Lambda(x)}-\sqrt{\frac{\omega^2_{x}(y)}{\Lambda^2(x)}+\frac{F^2(x, y)}{\Lambda(x)}}\right\}\subset A.
\ealn 

In particular, in~this case  both $T^A$ and $T^{-, A}$ are convex and $L$ is smooth on $N=\big\{\big((t,x), (\tau,y)\big)\in T\tilde M\setminus 0:L\big((t,x), (\tau,y)\big)=0\big\} = \big\{\big((t,x), (\tau,y)\big)\in T\tilde M\setminus 0:\tau=\frac{\omega_{x}(y)}{\Lambda(x)}\pm\sqrt{\frac{\omega^2_{x}(y)}{\Lambda^2(x)}+\frac{F^2(x, y)}{\Lambda(x)}}\big\}$.
The vector field $\partial_t$ is a timelike Killing vector field of $(\tilde M, L)$ which  is static if $\omega=0$ (with integral  manifolds $\{t\}\times M$, $t\in\R$). Finsler spacetimes $( \R\times M, L)$, with~$L$ of the type \eqref{stationary} and $\omega=0$ have been called in~\cite{CapSta16}, {\em standard~static Finsler spacetime}. \vspace{-3pt}
\begin{Remark}\label{standardtostatic}
	We observe that the slit cone sub-bundle $T$ is defined also as the set of  timelike vectors with positive component $\tau$ of the standard static Finsler spacetime  $(\R\times M, L_\omega)$, where \vspace{-3pt}
	\beq\label{static}
	L_\omega\big((t,x), (\tau,y)\big):=-\tau ^2+F_\omega^2(x,y),\eeq
	and $F_\omega$ is  given by
	\beq\label{Fomega}
	F_\omega(x,y)=\frac{\omega_{x}(y)}{\Lambda(x)}+\sqrt{\frac{\omega^2_{x}(y)}{\Lambda^2(x)}+\frac{F^2(x, y)}{\Lambda(x)}}.\eeq
	
	In fact, from~\cite{CapSta18}, Th. 5.1, $F_\omega$ is a Finsler metric  on $M$.
\end{Remark}
\section{Vacuum Field~Equations}
In general relativity, geodesics deviation equation is used to describe  the  relative acceleration of a congruence of point particles.  In~particular, in~vacuum, the~absence of tidal forces implies that  $R^i_{jil}y^jy^l=0$, where $R^i_{jkl}$ are the components of the Riemann curvature tensor; as a consequence  Einstein field equations $R_{jl}=R^i_{jil}=0$ are satisfied and, vice~versa, if~$R_{jl}=0$, then $R^i_{jil}y^jy^l=0$. In~\cite{Rutz93}, S. Rutz used this equivalence to generalize  Einstein vacuum  field equations to the Finsler  setting  as a single scalar equation $R(x,y)=R^i_i(x,y)=0$ on the slit tangent bundle. 
Here $R=R(x,y)$ is the {\em Finsler Ricci scalar} defined as follows.  
Let $\tilde g^{ij}(x,y)$ the components of the inverse of the matrix representing the fundamental tensor $\tilde g$ at the point $(x,y)\in A$ and let $G^i(x,y)$, $(x,y)\in A$,   be the {\em spray coefficients of $L$}:
\beq\label{Gi}
G^i(x,y):=\frac 1 4\tilde g^{ij}(x,y)\left(\frac{\partial^2L}{\partial x^k\partial y^j}(x,y)y^k-\frac{\partial L}{\partial x^j}(x,y)\right),\eeq
so that a smooth curve $\gamma$, such that $(\gamma,\dot\gamma)\subset A$, is a geodesic of $L$ if and only if, in~natural local coordinate on $T\tilde M$, $\ddot\gamma^i+2G^i(\gamma,\dot\gamma)=0$. 
Let
\bml\label{Rij}
R^i_k(x,y):=2\frac{\partial G^i}{\partial x^k}(x,y)-y^m\frac{\partial^2 G^i}{\partial x^m\partial y^k}(x,y)+2G^m(x,y)\frac{\partial^2 G^i}{\partial y^m\partial y^k}(x,y)\\
-\frac{\partial G^i}{\partial y^m}(x,y)\frac{\partial G^m}{\partial y^k}(x,y).
\eml

The {\em Riemann curvature of $L$} at $(x,y)\in A$ is  the linear map $\mathbf R_y:T_xM\to T_xM$ given by
$\mathbf R_y:=R^i_k(x,y)\partial_{x^i}\otimes dx^k$. It can be shown (see \cite{Shen01}, Equations (8.11)--(8.12)) that $R^i_k(x,y)=R^i_{jkl}(x,y)y^jy^l$ 
where $R^i_{jkl}$ are the components of the $hh$ part of the curvature $2$-forms of the Chern connection which are equal, for~any $(x,y)\in A$, to~\vspace{-3pt}
\beq R^i_{jkl}(x,y):=\frac{\delta \Gamma^i_{jl}}{\delta x^k}(x,y)-\frac{\delta \Gamma^i_{jk}}{\delta x^l}(x,y)+\Gamma^m_{jl}(x,y)\Gamma^i_{mk}(x,y) -\Gamma^m_{jk}(x,y)\Gamma^i_{ml}(x,y),\label{Rijkl}\eeq
being $\frac {\delta }{\delta x^i}$ the vector field on $A$ defined by $\frac {\delta }{\delta x^i}:=\frac{\partial}{\partial x^i}-N^m_i(x,y)\frac{\partial }{\partial y^m}$, where $N^m_i(x,y):=\frac{\partial G^m}{\partial y^i}(x,y)$,
and $\Gamma^i_{jk}$ are the components of the Chern connection, 
\beq\label{Gamma}
\Gamma^i_{jk}(x,y):=\frac 1 2 \tilde g^{il}(x,y)\left(\frac{\delta \tilde g_{lk}}{\delta x^j}(x,y)-\frac{\delta \tilde g_{jk}}{\delta x^l}(x,y)+\frac{\delta \tilde g_{lj}}{\delta x^k}(x,y)\right),\eeq
for all $(x,y)\in A$.
The Finsler Ricci scalar is then the  contraction of the Riemann curvature $R(x,y):=R^i_i(x,y)$, $(x,y)\in A$. It has been observed in~\cite{HoPfVo19} that   Rutz's equation is not variational but can be completed, in~a suitable sense,  to~a variational equation on $A\setminus N$ (which  coincides  with the field equation in~\cite{PfeWoh12a}  on the set  $\{(x,y)\in A:L(x,y)=-1\}$): 
\beq\label{fieldeq}
\frac{3R}{L}-\frac{1}{2}\tilde g^{ij}\frac{\partial^2R}{\partial y^i\partial y^j}    - \tilde g^{ij}\left(\frac{\delta P_{i}}{\delta x^j}-P_h\Gamma^h_{ij}- P_iP_j + \frac{\partial}{\partial y^j} \Big(y^k\big(\frac{\delta P_{i}}{\delta x^k}-P_h\Gamma^h_{ik}\big)\Big)\right) = 0
\eeq
where  
\[P^i_{jk} := \frac{\partial ^2G^i}{\partial y^j\partial y^k} -\Gamma^i_{jk}\]
are the components of the {\em Landsberg tensor} and $P_i=P^l_{li}$.
We stress that both equations $R=0$ and \eqref{fieldeq} are equivalent to Einstein vacuum equation $\mathrm{Ric}(h)=0$ if $L$ comes from a Lorentzian metric $h$, $L(x,y)=h_x(y,y)$ (see, respectively,  \cite{Rutz93}, \S 3 and \cite{HoPfVo19}, \S VII).
\section{On the Analyticity of  the Average Metric of a Static Berwald~Solution}\label{sec5}
We consider now a static Finsler spacetime $\tilde M=\R\times M$ with $L$ of the type \eqref{static}, but~$F_\omega$ will be any Finsler metric $F$ on $M$, not necessarily the one in \eqref{Fomega}.
Let us assume also that $F$ is a {\em Berwald metric}. This~means that    the   components  of the Chern connection of $F$ (defined as in \eqref{Gamma} with the fundamental tensor $g$ of $F$ replacing $\tilde g$) do not depend on $\br (x,y)\er\in TM$ or equivalently the components  $N^i_j\br (x,y)\er$ are linear in $y$ (precisely, it holds $N^i_j\br (x,y)\er =\Gamma^i_{jk}(x)y^k$, see \cite{BaChSh00}, prop. 10.2.1). From~\eqref{Rijkl},  the~ components of the Riemannian curvature tensor $R^i_{jkl}$ of $F$ are independent of $y$ too and the  Finsler Ricci scalar is equal to $\br R(x,y) \er=R^i_{jil}(x)y^jy^l$ for any $(x,y)\in TM\setminus 0$. 

Let us use the index  $0$ for the components corresponding to the coordinate $t\in\R$ and by $\alpha,\beta,\br \gamma\er$ the ones corresponding to coordinate systems in  $M$; moreover let us distinguish  Finslerian quantities  of $(\tilde M,L)$ from the ones of $(M,F)$ by a tilde. It can be soon realized that $L$ is Berwald as well; \br indeed as $\tilde g^{00}= -1$ and $\tilde g^{0\alpha}= 0$, from~\eqref{Gi}, taking also into account that  $L$ does not depend on $t$ ($:=x^0$) and $\frac{\partial L}{\partial x^\alpha}=\frac{\partial F^2}{\partial x^\alpha}$, we~ get \er $\tilde G^0=0$ and $\tilde G^{\alpha}((t,x),(\tau,y))=G^\alpha(x,y)$. Thus, $\tilde N^0_i=0$, $\tilde N^\alpha_0=0$ and $\tilde N^\alpha_\beta((t,x),(\tau,y))=N^\alpha_\beta(x,y)$, i.e.,~they are all linear in $(\tau,y)$.

Since $L$ is Berwald,  its non-vanishing spray coefficients 
\bmln 
\tilde G^\alpha\big((t,x),(\tau,y)\big)=\frac 1 2\left( \tilde N^\alpha_0\big((t,x),(\tau,y)\big)\tau+\tilde N^\alpha_{\beta}\big((t,x),(\tau,y)\big)y^\beta\right)=\\ 
\frac 1 2N^\alpha_{\beta}(x,y)\big)y^\beta=G^\alpha(x,y)y^\beta
\emln
are quadratic in $y$ and then, as~in \cite{Shen01}, Prop. 7.2.2, we get $\frac{\partial ^2\tilde G^\alpha}{\partial y^i\beta\partial y^k} =\tilde \Gamma^\alpha_{ik}$. Since  $\tilde g^{00}=\tilde g_{00}=-1$, $\tilde g^{0,\alpha}=\tilde g_{0\alpha}=0$ and $\frac{\delta \tilde g_{jk}}{\delta x^0}=0$, for~all $j,k\in\{0,\ldots n \}$, we also have that $\tilde \Gamma^0_{jk}=0$, for~all $j,k\in\{0,\ldots n \}$. Thus,    the~  Landsberg tensor $\tilde P^i_{jk}$ vanishes. Hence, for~Berwald $L$, \eqref{fieldeq} reduces to  \vspace{-3pt}
\begin{equation*}
\frac{3\tilde R}{L}-\frac{1}{2}\tilde g^{ij}\frac{\partial^2\tilde R}{\partial y^i\partial y^j}=0.
\end{equation*}
Taking into account that $\frac{\partial^2\tilde G^\alpha}{\partial y^0\partial y^k}=\frac{\partial^2 G^\alpha}{\partial y^0\partial y^k}=0$, for~all $k\in\{0,\ldots, n\}$,  this also implies that $\tilde \Gamma^\alpha_{0k}=0$ and $\tilde \Gamma^\alpha_{\beta\gamma}(t,x)=\Gamma^\alpha_{\beta\gamma}(x)$, that could be also proved directly by \eqref{Gamma}.  Thus, $\tilde R^0_{j0l}=0$, $\tilde R^{\alpha}_{0\alpha l}=0=\tilde R^{\alpha}_{j\alpha 0}$ and $\tilde R^{\alpha}_{\beta\alpha\gamma}\big((t,x),(\tau,y)\big)=\tilde R^{\alpha}_{\beta\alpha\gamma}\big(t,x)= R^{\alpha}_{\beta\alpha\gamma}(x)$, which imply that 
the Finsler Ricci scalar  of $F$ and $L$ coincide.  Thus,    if~$F$, or~equivalently $L$, has  vanishing Ricci scalar,  \br $0=R(x,y)=\tilde R\big((t,x),(\tau,y))$ for all  $\big ((t,x),(\tau, y)\big) \in T\tilde M\setminus 0$ then $L$  satisfies Equation \eqref{fieldeq} 
on $T\tilde M\setminus (\mathcal T\cup N)$.\er\footnote{The converse is true only in some cases, for~example for some Bogoslowski-Berwald metric, see~\cite{FuPaPf18}.}

It is well-known that the components \eqref{Gamma} of the Chern connection of a Berwald metric can be obtained from different Riemannian metrics as their  Christoffel symbols,~\cite{Szabo81}.
In particular,  see~\cite{Crampi14}, this   is true for the Riemannian metric 
\beq h_x(V_1,V_2):=\frac{\int_{S_x} g(x,y)[V_1,V_2]d\lambda(y)}{\int_{S_x}d\lambda(y)},\label{riemannian}\eeq
where 
$S_x:=\{y\in T_xM:F(y)=1\}$, $x\in M$   and  $d\lambda(y)$ is  the measure induced on $S_x$  by the Lebesgue measure on $\R^n$.

Let us assume that $F$ is of class $C^4$ on $TM\setminus 0$; then  $g(x,y)$ is of class $C^2$  on $TM\setminus 0$. Notice that   
the indicatrix bundle $\{(x,y)\in TM: F(x,y)=1\}$ is a $C^4$ embedded hypersurface in $TM$.  Thus,  both the area  of $S_x$ and the numerator in \eqref{riemannian} are  $C^2$ in $x$ and then $h$ is a $C^2$ Riemannian metric on $M$.
From  \eqref{Rijkl} and the fact that $F$ is Berwald, the~components $R^i_{jkl}$ are  equal to the ones of the Riemannian curvature tensor of $h$ and then  we have \vspace{-3pt} \[\mathrm{Ric}(h)_{\alpha\beta}(x)=R^m_{\alpha m\beta}(x)=\frac 1 2 \frac{\partial^2}{\partial y^\alpha\partial y^\beta}\left (R^m_{jml}(x)y^jy^l\right)=\frac 1 2 \frac{\partial^2 R}{\partial y^\alpha\partial y^\beta}(x,y)=0.\]

From \cite{DeTKaz81}, Theorem 4.5, it follows that 
in an atlas of $M$ of harmonic coordinates of $h$,   $h$ itself is analytic.
We can summarize the above reasoning in the following result:\vspace{-3pt}
\begin{Theorem}\label{main}
	Let $(\tilde M, L)$ be a standard static Finsler spacetime, $\tilde M=\R\times M$, $L\big((t,x),(\tau,y)\big)=-t^2+F^2(x,y)$. Assume that $F\in C^4(TM\setminus 0)$ is Berwald with vanishing Ricci Finsler scalar $R$, then $L$ is Berwald, satisfies the field equation \eqref{fieldeq}
	with  and  the metric  $\frac{\int_{S_x} g(x,y)[\cdot, \cdot]d\lambda(y)}{\int_{S_x}d\lambda(y)}$ on $M$, where $g$ is the fundamental tensor of the Finsler metric $F$, is~analytic in its harmonic coordinates. \vspace{-3pt}
\end{Theorem}                                                                                                                                                                                                                                                                                                                                                                
Let us now consider the case when $F$ is equal to $F_\omega$ in \eqref{Fomega}; in light of Theorem~\ref{main} we would like to have conditions  ensuring that $F_\omega$ is Berwald. It is well-known that for a Randers metric $F=\alpha +\beta$, where $\alpha$ is the norm of a Riemannian metric and $\beta$ a one-form, this holds if and only if $\nabla \br \beta\er=0$, where $\nabla$ is the Levi--Civita connection of the Riemannian metric (see \cite[Th. 11.5.1]{BaChSh00}). Let us see that a sufficient condition of this type holds for $F_\omega$ as well. Let us denote by $\beta$ the one-form on $M$ defined as $\beta :=\omega/\Lambda$ and by   $G$ the Finsler metric given by $G:=\big(F^2/\Lambda+ \beta^2\big)^{1/2}$.\vspace{-3pt}
\begin{Proposition}\label{enough}
	Assume that the Finsler metric $F/\sqrt{\Lambda}$ on $M$ is Berwald with vanishing Ricci scalar and that $D\beta=0$, where $D$ is the linear covariant derivative on $M$ induced by the Chern connection of $F/\sqrt{\Lambda}$. Then the Finsler metric $G+\beta$ is Berwald with vanishing Ricci scalar as well.\vspace{-3pt}
\end{Proposition}	
\begin{proof}
	Let us show firstly that $G$ is Berwald. 
	In order to evaluate the spray coefficients of $G$, we  compute the geodesics equation of the Finsler manifold $(M,G)$ as the Euler--Lagrange equation of  the energy functional $E_G$ of $G$. Without~loosing generality, we can  assume for this purpose that $\gamma\colon[a,b]\to M$, $\gamma=\gamma(s)$, is a smooth regular curve (i.e.,  $\dot\gamma(s)\neq 0$, for~all $s\in [a,b]$) and that $\sigma\colon   [-\epsilon,\epsilon ]\times [a,b]\to M$, $\sigma=\sigma(r,s)$, is a smooth variation of $\gamma$ (i.e., $\sigma(0,\cdot)=\gamma$) such that for all $r\in [-\epsilon,\epsilon]$, $\sigma(r,a)=\gamma(a)$, $\sigma(r,b)=\gamma(b)$ and $\sigma(r,\cdot)$ is regular as well. Let us denote respectively by $T$ and $U$ the vector field along $\sigma$ defined by $\partial_t\sigma$ and $\partial_r\sigma$. Now, in~order to  compute the variation of $E_G$ associated to $\sigma$,  we can consider  separately the terms coming from the variation of the energy functional of $F/\sqrt{\Lambda}$ and the ones coming from the variation  of $\frac 12 \int_a^b \beta^2(\dot\gamma)d s$. The~variation of the latter functional is equal to 
	\beq\label{varbeta2}\frac 1 2\int_a^b\partial_r\Big(\beta^2(T)\Big)ds =\int_a^b\beta(T)\Big(\big(D_U\beta\big)(T)+\beta(D_UT)\Big)ds =\int_a^b  \beta(T)\beta(D_TU)ds \eeq
	where we have used the fact that the connection $D$ is torsion free (see \cite[p. 262]{BaChSh00}) and hence $D_UT=D_TU$. 
	Evaluating \eqref{varbeta2}  at $r=0$ gives 
	\bmln \int_a^b\beta(\dot\gamma)\beta\big(D_{\dot\gamma} V \big)ds=\int_a^b\beta(\dot\gamma)\frac{d}{ds}\beta(V)ds=\\ -\int_a^b\beta(V)\frac{d}{ds}\beta(\dot\gamma)ds=-\int_a^bg(\gamma,\dot\gamma)[B,V]\frac{d}{ds}\beta(\dot\gamma)ds,\emln
	where $V$ is the variational vector field associated to $\sigma$, i.e.,~$V:=U\big(\sigma(0,\cdot)\big)=\partial_r\sigma(r,\cdot)|_{r=0}$ and  $B$ is the vector field along $\gamma$ representing $\beta$ with respect to the Riemannian metric, over~$\gamma$,  $g(\gamma,\dot\gamma)[\cdot,\cdot]$,  $g$ being the fundamental tensor of $F/\sqrt{\Lambda}$. 
	As the variation of the energy functional of $F/\sqrt{\Lambda}$ at $r=0$ gives $\int_a^b g(\gamma,\dot\gamma)[\dot\gamma, D_{\dot\gamma}V]ds$, we get that a 
	a smooth critical point $\gamma$ of $E_G$ satisfies the equation
	\beq \label{fakeeq}
	D_{\dot\gamma}\dot\gamma+B\frac{d}{ds}\beta(\dot\gamma)=0\eeq
	
	Hence
	\[0=\beta(D_{\dot\gamma}\dot\gamma)+\beta(B)\frac{d}{ds}\beta(\dot\gamma)=\frac{d}{ds}\beta(\dot\gamma)+\beta(B)\frac{d}{ds}\beta(\dot\gamma)=\frac{d}{ds}\beta(\dot\gamma)\big(1+\beta(B)\big)\]
	
	As $\beta(B)=g(\gamma,\dot\gamma)[B,B]\geq 0$, we get  $\frac{d}{ds}\beta(\dot\gamma)=0$ and hence
	$\gamma$ satisfy the equation $D_{\dot\gamma}\dot\gamma=0$. This~implies that the spray coefficients of $G$ are quadratic in the velocities and then $G$ is Berwald.
	Let us now prove that also $G+\beta$ is Berwald. To~this end, let us compute the  variation of the length functional $\ell$ of  $G+\beta$. As~above, let us consider a smooth regular curve $\gamma$. Since $\ell$ is invariant under orienting preserving reparametrization, we can assume that $\gamma$ is parametrized w.r.t the arch length of $G$, i.e.,~$G(\gamma,\dot\gamma)=1$. Let $l$ be the length of $\gamma$ w.r.t. $G$. Arguing as above, the~first variation of $\ell$ at $r=0$ is equal to\vspace{-3pt}
	\[\int_0^l\beta(D_\gamma V)d s+\int_0^l\Big(g(\gamma,\dot\gamma)[\dot\gamma, D_{\dot\gamma}V]-g(\gamma,\dot\gamma)[B,V]\frac{d}{ds}\beta(\dot\gamma)\Big)ds.\] 
	
	The first integral above is equal to $ \int_0^l\frac{d}{ds}\beta(V)d s$ and hence it vanishes for all variational vector fields $V$. Therefore, the~critical points of $\ell$ parametrized w.r.t. to the arc length of $G$ satisfies  \eqref{fakeeq} and then, as~above, they do satisfy equation $D_{\dot\gamma}\dot\gamma=0$. Since $\frac{d}{ds}\beta(\dot\gamma)=0$, $\beta(\dot\gamma)=\mathrm{const.}$, i.e.,~$\gamma$ is also affinely parametrized for the metric $G+\beta$, hence it is a geodesic of this metric.   This implies that $D_{\dot\gamma}\dot\gamma=0$ is also the geodesics equation of $G+\beta$ and the  spray coefficients of this metric  are equal to the ones of $F/\sqrt{\Lambda}$. As~a consequence, $G+\beta$ is Berwald and its Finsler Ricci scalar vanishes because it is equal to the Finsler Ricci scalar of $F/\sqrt{\Lambda}$ (recall \eqref{Rij}).
\end{proof}
\begin{Remark}
	A similar proof shows that  $D\beta=0$ is also a sufficient condition for a Randers variation of a Finsler metric $F$ of Berwald type (i.e., for~a Finsler metric of the type $F+\beta$ with  $F(x,y)+\beta_x(y)>0$ for all $(x,y)\in TM\setminus 0$)   to be Berwald as well. This extends beyond the case that $F$ is Riemannian the sufficient condition in \cite{BaChSh00}, Th. 11.5.1.  
\end{Remark}
\section{Conclusions}
We have reviewed the mathematical definitions of a Finsler spacetime  and of a static timelike Killing vector field on it,  based on a fundamental function $L$ with low regularity assumptions.  In~particular, we have relaxed the requirement in~\cite{HoPfVo19} about smoothness of  $L$ on the open cone sub-bundle defining admissible timelike future-oriented vectors, in~order to include static and stationary Finsler spacetimes that split as a product $\R\times M$. 
We have then considered Berwald static Finsler spacetime under the point of view of the
analyticity of the solutions of the Rutz's equation $\tilde R(x,y)=0$ (and then satisfying also the field equation  \eqref{fieldeq} proposed in~\cite{HoPfVo19}). 
We have obtained a partial result in this direction stating analyticity \br (in its harmonic coordinates) \er of any Riemannian metric whose Levi--Civita connection coincide with  the Chern connection of the Finsler metric on the base $M$. 
In particular, this holds for the metric \eqref{riemannian} obtained as an average of the fundamental tensor of the Finsler metric on the base $M$.

\br The existence of analytic solutions (in a fixed coordinate system) \er of the Rutz's equation has been recently obtained for Berwald Finsler pp-waves in~\cite{FusPab16}  introduced there (see also  \cite{GomMin18}, \S 4).    
The Berwald static case that we have considered is, on~the other hand, dynamical equivalent to the classical Lorentzian static case, at~least when the dynamic is governed by the  Rutz's equation. Nevertheless,  extending   Theorem~\ref{main} to more general classes of Finsler function $F$ seems difficult due to the lack of ellipticity and quasi-diagonality of the system of equations $R_{\br \alpha\beta \er}=R^l_{\br \alpha\er l\br \beta \er}=0$ (that could be considered instead of the scalar equation $R=0$, see~\cite{Rutz93}, \S 3), even writing it in harmonic coordinates w.r.t. the horizontal Laplacian  of $F$, see~\cite{CapMas19}, Remark~5. \br From this point of view, it might be interesting to analyse  a generalization  of the  Einstein field equations  on the whole  tangent bundle of the spacetime, obtained recently~\cite{TriSta18}, based on Sasaki type metrics and  nonlinear connections on it. \er

\vspace{6pt} 	 	





\end{document}